\documentclass[leqno,a4paper, 12pt]{article}

\usepackage{amsmath,amssymb,amsthm}
\usepackage[polish,english]{babel}
\usepackage{polski}
\usepackage[cp1250]{inputenc}
\usepackage[T1]{fontenc}
\usepackage{makeidx}
\newtheorem{rem}{Remark}[section]
\newtheorem{coll}{Corollary}[section]
\newtheorem{thm}{Theorem}[section]
\newtheorem{lem}{Lemma}[section]

\newtheorem{defin}{Definition}[section]
\newtheorem{propos}{Proposition}[section]

\newcommand{\bP}{\mathbb{P}}

\newcommand{\ve}{\varepsilon}
\makeindex
\begin{document}
\section*{\centerline{Stability of measure solutions to a generalized Boltzmann} \linebreak
\centerline{equation with collisions of a random number of particles $^{\diamond}$}}
\vskip5mm
\centerline{Henryk Gacki$^{a}$ and Łukasz Stettner$^{b, \star}$}
\vskip4mm
\centerline{$^{a}${\small{\it Faculty of Science and Technology, University of Silesia in Katowice, Bankowa 14,}}}
 \centerline{\small{\it 40-007 Katowice, Poland}}
 \vskip2mm
\centerline{$^{b}${\small{\it Institute of Mathematics Polish Academy of Sciences, Śniadeckich 8,}}}
\centerline{\small{\it 00-656 Warsaw, Poland }}


\vskip0.8cm

\section*{Abstract}
In the paper we study a measure version of the evolutionary nonlinear Boltzmann-type equation in which we admit  a random number of collisions of particles. We consider first a stationary model and use  two methods to find its fixed points: the first based on Zolotarev seminorm and the second on Kantorovich-Rubinstein  maximum principle. Then a dynamic version of Boltzmann-type equation is considered and its asymptotical stability is shown.
\\

{\it Keywords}: Boltzmann equation,  collisions of random particles,  Zolotariev seminorm,  \linebreak Kantorovich-Rubinstein maximum principle\\

{\it 2010 MSC}:  82B31,	82B21.

\section{{\bf Introduction}}\label{sec1Int}

In the paper we consider a nonlinear evolutionary measure valued Boltzmann type equation of the form
 \begin{equation}\label{s4.w1.11adodm2}
    \frac{d\psi}{dt} + \psi =\bP\,\psi \qquad\text{for}\qquad t \geq 0
\end{equation}
where the operator $\bP$  maps  $\mathcal{M}_{1}(\mathbb R_{+})$ the space of probability measures on $\mathbb{R}_{+}=[0,\infty)$ into itself. We are looking for $ \psi : \mathbb{R}_{+} \rightarrow \mathcal{M}_{sig}(\mathbb{R}_{+})$ with $\psi_0\in\mathcal{M}_1(\mathbb{R}_{+})$, where  $\mathcal{M}_{sig}(\mathbb{R}_{+})$ (or shortly $\mathcal{M}_{sig}$) is the space of all signed measures on $\mathbb{R}_{+}$.

Equation \eqref{s4.w1.11adodm2} is a generalized version of the equation considered in \cite{tjon wu} (see also section 8.9 in \cite{LM}, or \cite{bob} for the motivation),
\begin{equation}\label{8.9.2}
\frac{\partial u(t,x)}{\partial t}+
u(t,x)=\int\limits_{x}\limits^{\infty}\frac{dy}{y}\int\limits_{0}\limits^{y}u(t,
y-z)u(t,z)dz:=P\,u(x) \qquad t\geq 0,\qquad x\geq 0,
\end{equation}
which describes energy changes subject to the collision operator $ P\,u(x)$ and which was obtained

------------------------------------------------------------------------------------------------------------------------

{} $^{\diamond}$  {\small{Research supported by National Science Center, Poland by grant UMO- $2016/23/B/ST1/00479$}}

{} $^{\star}$ {\small{Corresponding author}}

{} {\small{{\it Email addresses}: henryk.gacki@us.edu.pl (Henryk Gacki), stettner@impan.pl (Łukasz Stettner)}}

\noindent
from Boltzmann equation corresponding to a spatially homogeneous gas with no external forces using  Abel transformation. To be more precise, in the theory of dilute gases Boltzmann equation in the general form
${ dF(t,x,v)\over  dt}=C(F(t,x,v))$ gives us an information about time, position and velocity of particles of the dilute gas. This equation is a base for many mathematical models of colliding particles. In particular, for a spatially homogeneous gas we come to the equation \eqref{8.9.2} with additional conditions saying that its solution $u$, for fixed $t$,  is a density with first moment equal to $1$, which in turn corresponds to the conservation law of mass and energy.
The operator $Pu$ is a density function of the random variable $\eta (\xi_{1} +  \xi_{2})$ where random variables $\eta$, $\xi_{1}$ and $\xi_2$ are independent and $\eta$ is uniformly distributed, while $\xi_1$, $\xi_2$ have the same density function $u$.
The assumption that $\eta$ has uniform distribution on $[0,1]$  is quite restrictive
and there are no physical reasons to assume that the distribution of energy of particles can be described only by its density (is absolutely continuous).
Moreover collision of two particles maybe replaced by collision of a random number of particles. This is a reason that in what follows we shall consider a measure valued version
 \eqref{s4.w1.11adodm2} of the equation \eqref{8.9.2}.

Let
\begin{equation}\label{s4.w1.20adodm}
D := \big\{\mu\in \mathcal{M}_{1}(\mathbb R_{+}) : m_{1}(\mu) = 1  \big\},\quad\text{with}\quad m_{1}(\mu) = \int\limits_{0}\limits^{\infty}\,x\mu(dx),
\end{equation}
and denote by $\bar{D}$ a weak closure of $D$, which is of the form
 \begin{equation}\label{s4.w1.20adod}
\bar{D} := \big\{\mu\in \mathcal{M}_{1}(\mathbb R_{+}) : m_{1}(\mu) \leq 1  \big\}.
\end{equation}
The operator $P$ defined in \eqref{8.9.2}  describes collision of two particles. In what follows we shall consider a general situation of collision of a random number of particles.
To describe the collision operator in this case we start from recalling  the convolution operator of order $n$ and the linear operator $P_{\varphi}$,  which is related to multiplication of random variables.

For every $n \in \mathbb{N}$ let $P_{*n}:\mathcal{M}_{sig} \rightarrow \mathcal{M}_{sig}$, be given by the formula
\begin{equation}\label{equ63}
P_{*1} \mu :=\mu, \quad P_{*(n+1)} \mu:=\mu  * P_{*n}\mu \quad\text{for}\quad \mu\in \mathcal{M}_{sig}.
\end{equation}

It is easy to verify that $P_{*n}(\mathcal{M}_{1}(\mathbb R_{+}))\subset \mathcal{M}_{1}(\mathbb R_{+})$ for every $n\in \mathbb{N}$. Moreover, $P_{*n}|_{\mathcal{M}_{1}(\mathbb R_{+})}$ has a simple probabilistic interpretation: if $\xi_1,...,\xi_n$ are independent  random variables with the same probability distribution $\mu$, then $P_{*n}\mu$ is the probability distribution of $\xi_1+...+\xi_n$. \\

The second class of operators we are going to study is related to multiplication of random variables. Formal definition is as follows: given $\mu,v \in \mathcal{M}_{sig}$, we define product $u \circ  v$ by
\begin{equation}
(\mu \circ  v)(A):= \int_{\mathbb{R}_+} \int_{\mathbb{R}_+} \mathbf{1}_{A} (xy) \mu(dx)v(dy) \quad\text{for}\quad A\in {\mathcal{B}_{\mathbb{R}_+}}.
\end{equation}
and
\begin{equation}\label{equ58}
\big \langle f,\mu  \circ  v \big \rangle =  \int_{\mathbb{R}_+} \int_{\mathbb{R}_+} f(xy) \mu(dx)v(dy)
\end{equation}
for every Borel measurable $f:\mathbb{R}_+ \rightarrow \mathbb{R}$ such that $(x,y) \mapsto f(xy)$ is integrable with respect to the product of $|\mu|$ and $|v|$. For fixed $\varphi \in \mathcal{M}_{1}$ define
\begin{equation}\label{equ62}
P_{\varphi}\mu := \phi \circ \mu \quad\text{for}\quad \mu\in \mathcal{M}_{sig}.
\end{equation}

 Similarly as in the case of convolution it follows that $P_\varphi (\mathcal{M}_1) \subset \mathcal{M}_1$. For $\mu\in \mathcal{M}_1$ the measure $ P_\varphi \mu$ has an immediate probabilistic interpretation: if $\varphi$ and $\mu$ are probability distributions of random variables $\xi$ and $\eta$ respectively, then $P_{\varphi}\mu$ is the probability distribution of the product  $\xi \eta$.

We introduce now definition of more general version of $P$ allowing infinite number of collisions:
\begin{equation}\label{equ61}
\bP:=\sum_{i=1}^\infty \alpha_i P_{\varphi_i} P_{*_i}=\sum_{i=1}^\infty \alpha_i \bP_i,
\end{equation}
where we have  $\bP_i:=P_{\varphi_i} P_{*_i}$, $\sum\limits_{i=1}\limits^\infty \alpha_i=1$, $\alpha_i\geq 0$,  $\varphi_i \in \mathcal{M}_1 $ and $m_1 (\varphi_i)=1/i$ \linebreak and the limit of the series is considered in the weak topology sense that is, $\sum\limits_{i=1}\limits^n \alpha_i P_{\varphi_i} P_{*_i}\Rightarrow \sum\limits_{i=1}\limits^\infty \alpha_i P_{\varphi_i} P_{*_i}$, which means
  $\sum\limits_{i=1}\limits^n \alpha_i P_{\varphi_i} P_{*_i}(f)\to \sum\limits_{i=1}\limits^\infty \alpha_i \bP_i(f)$ as $n\to \infty$, for any continuous bounded function $f$ defined on $\mathbb{R}$.  From (\ref{equ61}) it follows that $\bP \mathcal{M}_1 \subset \mathcal{M}_1 $. Using (\ref{equ63}) and (\ref{equ62}) it is easy to verify that for $\mu \in D$,
\begin{equation}
m_1 (P_{*i} \mu)=i \quad\text{and}\quad m_1 (P_{\varphi_i} \mu) =1/i.
\end{equation}

Given $\mu \in \mathcal{M}_1$ the value of $\bP\mu$ can be considered as the probability  distribution  of a random variable $\zeta$ defined as
\begin{equation}\label{s4.w1.6adodm}
\zeta:=\eta_\tau (\sum_{j=1}^\tau \xi_{\tau j}),
\end{equation}
where we have  sequences of
independent random variables $\eta_i$, $\xi_{ij}$, $j=1,2,\ldots,$   and $\tau$, such
that $\xi_{ij}$ have the same probability distribution $\mu$ for $j=1,2,\ldots,$, random variables
$\eta_i$ have the probability distribution  $\varphi_i$ and random variable $\tau$ takes values in the set $\left\{1,2,\ldots\right\}$ such that $P\left\{\tau=j\right\}=\alpha_j$.
 Physically this means that the number of colliding particles is random and energies of particles before a
collision are of independent quantities and that a particle after
collision of $i$-th particles takes the $\eta_i$ part of the sum of the energies of the
colliding particles.

We can also define $\tilde{\eta}_i:=i \eta_i$ and consider

\begin{equation}\label{defzeta}
\zeta:=\tilde{\eta}_\tau ({1\over \tau}\sum_{j=1}^\tau \xi_{ \tau j})
\end{equation}
and write
\begin{equation}\label{impf}
\bP=\sum_{i=1}^\infty \alpha_i \tilde{P}_{\varphi_i} \tilde{P}_{*_i}
\end{equation}
where $\tilde{P}_{\varphi_i}={P}_{\tilde{\varphi}_i}$ with $\tilde{\varphi}_i$ being the probability distribution of $i \eta_i$, and $ \tilde{P}_{*_i}\mu$ is the probability distribution of ${\xi_{i1}+\ldots + \xi_{ii} \over i}$.

Measure valued solutions to the Boltzmann-type equations with different collision operator and even in multidimensional case were studied in a number of papers see e.g. \cite{lm}, \cite{abcl} and \cite{Mor}. In the paper \cite{lm} existence and stability of measure solutions to the spatially homogeneous Boltzmann equations that have polynomial and exponential moment production properties is shown. In the paper \cite{abcl} existence and uniqueness of measure solutions to one dimensional Boltzmann dissipative equation and then their asymtotics is considered. In this paper first stationary (steady state) equation for a specific collision operator is studied and then dynamic fixed point theorem is used.
Asymptotics of solutions to the Boltzmann equation with infinite energy to so called self similar solutions was studied in \cite{bc}.
Asymptotic property of self similar solutions to the Boltzmann Equation for Maxwell molecules was then shown in \cite{bct}.
Long time behaviour of the solutions to the nonlinear Boltzmann equation for spacially uniform freely cooling inelastic Maxwell molecules was studied in \cite{bict}. Stability of Boltzmann equation with external potential was also considered in \cite{W} and in the case of exterior problem in  \cite{Y}.
Solutions of the Boltzmann equation with collision of $N$ particles and their limit behaviour when $n\to \infty$ over finite time horizon were studied in \cite{ap}. In \cite{bcg}, the N-particle model, which includes multi-particle interactions was considered. It is shown that under certain
natural assumptions we obtain a class of equations which can be considered as the most general Maxwell-type model.

In \cite{Rudzwol} an individual based model  describing phenotypic evolution
in hermaphroditic populations which includes random and assortative mating of individuals is introduced. By increasing the number of individuals to infinity a nonlinear transport equation is obtained, which describes the evolution of phenotypic probability distribution. The main result of the paper is a theorem on asymptotic stability of trait (which concerns the model with more general operator $P$) with respect to Fortet-Mourier metric.

Stability problems of the Boltzmann-type equation \eqref{s4.w1.11adodm2} with operator $P$ corresponding to collision of two particles was studied in the paper \cite{1 lasota traple}. The case with infinite number of particles was considered in \cite{20 lasota traple} using Zolotariev seminorm approach. Properties of stationary solutions corresponding two collision of two particles were studied in \cite{lastr}.

In this paper we study stability of solutions to one dimensional Boltzmann-type equation \eqref{s4.w1.11adodm2} with operator $P$ of the form $\bP$ defined in \eqref{equ61}.
We show that if this equation has a stationary solution $\mu^*$, such that its support covers $\mathbb{R}_+$, then taking into account positivity of solutions to \eqref{s4.w1.11adodm2}-\eqref{equ61} we have its asymptotical stability in Kantorovich - Wasserstein metric to $\mu^*$.
We consider first stationary equation and look for fixed points of the operator $\bP$. For this purpose we adopt two methods to show the existence of fixed point of $\bP$ with the first moment equal to $1$. The first method is based on Zolotarev seminorm and in some sense simplifies the method used in \cite{20 lasota traple}. The second method is based on Kantorovich-Rubinstein maximum principle and generalizes the results of \cite{7} and then also of \cite{1 lasota traple} to the case of a random number of colliding particles. We also show several characteristics of fixed points of $\bP$. Results on the stationary equation are then used to study stability of the dynamic Boltzmann equation. The novelty of the paper is that we consider the Boltzmann-type equation \eqref{s4.w1.11adodm2} with random unbounded number of colliding particles and show its stability in Kantorovich - Wasserstein metric using probabilistic methods, generalizing former results of \cite{1 lasota traple}, \cite{20 lasota traple} and \cite{7}.  To improve readability of the paper an appendix is added to the paper where some important results, which are used in the paper, are formulated and their proofs are sketched.

\section{Properties of the operator $\bP$}
We study first several properties $\bP$.
\begin{propos} \label{P1} Operator $\bP$ transforms the set $D$ or $\bar{D}$ into itself.\linebreak
 It is continuous with respect to weak topology in $\mathcal{M}_{1}(\mathbb R_{+})$ i.e.
 whenever \linebreak $\mathcal{M}_{1}(\mathbb R_{+})\ni\mu_n\Rightarrow \mu$ we have $\bP \mu_n\Rightarrow \bP \mu$ as $n\to \infty$.
 Furthermore whenever $m_r(\varphi_i)=\int\limits_{\mathbb R_+} x^r \varphi_i(dx)<\infty$ and $m_r(\mu)=\int\limits_{\mathbb R_+} x^r \mu(dx)<\infty$ for $r\geq 1$, $i=1,2,\ldots$  then
 \begin{equation}\label{rmom}
 m_r(\bP\mu)\leq\sum_{i=1}^\infty \alpha_i m_r(\varphi_i) i^r m_r(\mu).
\end{equation}
 \end{propos}
 \begin{proof}
Note first that for each $\mu\in \mathcal{M}_{1}(\mathbb R_{+})$ we have that $P_{\varphi_i} P_{*_i}\mu \in \mathcal{M}_{1}(\mathbb R_{+})$ for $i=1,2,\ldots$ and therefore  $\bP\mu \in \mathcal{M}_{1}(\mathbb R_{+})$. Moreover for $\mu\in D$
\begin{equation}
m_1(P_{\varphi_i}P_{*_i}\mu)=\int_0^\infty \ldots \int_0^\infty x(y_1+y_2+\ldots+y_i)\varphi_i(dx)\mu(dy_1)\ldots\mu(dy_i)={1\over i}i=1
\end{equation}
so that $P_{\varphi_i}P_{*_i}\mu\in D$ and consequently $\bP:D \mapsto D$. Similarly $\bP:\bar{D} \mapsto \bar{D}$.
\linebreak For a given continuous bounded function $f:\mathbb R_{+}\to \mathbb{R}$ and $\mathcal{M}_{1}(\mathbb R_{+})\ni\mu_n\Rightarrow \mu$, as $n\to \infty$ we have
\begin{eqnarray}\label{conv1}
&& P_{\varphi_i} P_{*_i}\mu_n(f)=\int_0^\infty \ldots \int_0^\infty f(x(y_1+y_2+\ldots+y_i))\varphi_i(dx)\mu_n(dy_1)\ldots\mu_n(dy_i)\to \nonumber \\
&&\int_0^\infty \ldots \int_0^\infty f(x(y_1+y_2+\ldots+y_i))\varphi_i(dx)\mu(dy_1)\ldots\mu(dy_i)=P_{\varphi_i} P_{*_i}\mu(f)
\end{eqnarray}
as $n\to \infty$.
In fact, from continuity of
$$
(y_1,y_2,\ldots,y_i) \to \int_0^\infty f(x(y_1+y_2+\ldots+y_i))\varphi_i(dx),
$$
and weak convergence of the measures $\mu_n(dy_1)\ldots\mu_n(dy_i)\Rightarrow \mu(dy_1)\ldots\mu(dy_i)$, using \eqref{conv1} we immediately obtain that $\bP \mu_n\Rightarrow \bP \mu$ which is desired continuity property.
Now using $\zeta$ defined in \eqref{defzeta}, independency of random variables, as well as convexity we obtain
\begin{eqnarray}
 m_r(\bP\mu)&=&\int\limits_{\mathbb R_+} x^r \bP\mu(dx)=E\left[\zeta^r\right]=\sum_{i=1}^\infty \alpha_i E\left[\eta^r_i\right] i^r E\left[\left({1\over i}\sum_{j=1}^i \xi_{ \tau j}\right)^r\right] \nonumber \\
&\leq& \sum_{i=1}^\infty \alpha_i m_r(\eta_i) i^r m_r(\mu),
\end{eqnarray}
which completes the proof.
\end{proof}
We comment below the formula \eqref{rmom}
\begin{rem} Note that since $m_1(\varphi_i)={1\over i}$ we have $\int\limits_{\mathbb R_+} x \varphi_i(dx)\leq \left(\int\limits_{\mathbb R_+} x^r \varphi_i(dx)\right)^{{1\over r}}$  and consequently $m_r(\varphi_i)\geq {1\over i^r}$, so that in general we may not have that
$\sum\limits_{i=1}\limits^\infty \alpha_i m_r(\varphi_i) i^r < \infty$.

This sum is finite however when for example we know that for a sufficiently large $i$ we have that $\alpha_i m_r(\varphi_i)\leq {1\over i^{\beta}}$ with $\beta>1+r$. Finiteness of the sum above shall play an important role in the approach to study fixed points of $\bP$ with the use of Zolotariev seminorm (see section 3).
\end{rem}
In what follows we are interested to find a fix point of the operator $\bP$ in the set $D$. It is clear clear $\mu=\delta_0$ is a fixed point of $\bP$ in $\bar{D}$. Typical way to find a fixed point is to consider iterations $\bP\mu$ for $\mu\in D$ and since the measures $\left\{\bP\mu,\bP^2\mu, \ldots\right\}$ are tight, expect a limit to be a fixed point. However even when such iteration converges in a weak topology, its limit will be in $\bar{D}$ and it is not clear that such limit will be a fixed point. Note furthermore that when we have more than one particle, the operator $\bP$ is nonlinear and we can not use several techniques typical for linear operators.

\begin{rem} Notice that fixed point is not necessary in $D$ since the set $D$ is not closed in the weak topology. In fact, when $\mu_n\left\{{1\over n}\right\}={n\over n+1}$ and  $\mu_n\left\{{n}\right\}={1\over n+1}$, then $\mu_n\in D$ and $\mu_n\Rightarrow \mu:=\delta_0$ so that total mass of $\mu$ is concentrated at $0$, and $\int_0^\infty x \mu(dx)=0$.
Similary for $0< \alpha <1$, letting $\mu_n(1-\alpha)=1-{1\over n}$, $\mu_n(1)={1-\alpha \over n}$ and
$\mu_n(n)={\alpha \over n}$ we clearly have that $\mu_n\in D$ and $\mu_n\Rightarrow \delta_{1-\alpha}$, as $n\to \infty$. One can show that the closure $\bar{D}$ of the set $D$ in the weak topology consists of all probability measures $\nu \in \mathcal{M}_{1}(\mathbb R_{+})$ such that $m_1(\nu)\leq 1$.
Consider the following example: $P\left\{\varphi_i=0\right\}={n\over n+1}$,
$P\left\{\varphi_i=(n+1){1\over i}\right\}={1\over n+1}$ and $\mu([\Delta,\infty))=1$ with $m_1(\mu)=1$, $\Delta >0$ and fixed integer $n\geq 1$. Then the support of $\bP\mu$ consists of $0$ and the second part contained in $[\Delta({n+1\over i}), \infty)$ and inductively the support of $\bP^k\mu$ contains $0$ and its second part is contained in $[\Delta({(n+1)\over  i})^{k}, \infty)$, for positive integer $k$. Since $\bP^k\mu(0)={n\over n+1}$ it is clear that $\bP^k\mu\Rightarrow \delta_0$, as $k\to \infty$, so that limit of $\bP^k\mu$, which is also a fixed point of $\bP$, is not in $D$.

On the other hand, when $P\left\{\varphi_i=\delta\right\}={n\over (n+1)-i\delta}$ and $P\left\{\varphi_i=(n+1){1\over i}\right\}={1-i\delta\over (n+1)-i\delta}$ with $i\delta<1$ and  $\mu([\Delta,\infty))=1$ with $m_1(\mu)=1$, $\Delta >0$ we have that
$\cup_k supp(\bP^k\mu)$ is dense in $[0,\infty)$.

When support of $\varphi_i$ contains a sequence of positive real numbers converging to $0$, then  $\cup_{k} supp(\bP^k\mu)$ is dense in $[0,\infty)$ no matter what $\mu\in D$ is chosen.   Assume additionally that ${1\over i}\in supp \varphi_i$ for each $i=1,2,\ldots$. Then  $supp(\bP^k\mu)\subset supp(\bP^{k+1}\mu)$, so that we have an increasing sequence of closed sets $supp(\bP^k\mu)$ which cover the interval $(0,\infty)$.
\end{rem}

Next two Propositions show properties of the fixed point of $\bP$.

Let $\Lambda:=\left\{i: \alpha_i>0 \right\}$
\begin{propos}
Let $\mu\in {D}$ be a fixed point of $\bP$. When there are at least two elements of $\Lambda$ and $\tilde{\varphi}_i=\delta_1$ for $i \in \Lambda$ then either $\mu=\delta_1$ or $m_\beta(\mu)=\infty$ for any $\beta>1$. When $\mu=\delta_1$ then  $\tilde{\varphi}_i=\delta_1$ for $i \in \Lambda$.
\end{propos}
\begin{proof}
Using \eqref{impf} we have that the law of
$\tilde{\eta}_\tau({1\over \tau}\sum_{j=1}^\tau \xi_{\tau j})$,
when the law of $\xi_{ij}$ is $\mu$, is also $\mu$. Then for $\beta>1$ we have
\begin{equation}\label{imineq0}
E\left\{\tilde{\eta}_\tau^\beta({1\over \tau}\sum_{j=1}^\tau \xi_{\tau j})^\beta\right\}=\sum_{i=1}^\infty \alpha_i m_\beta(\tilde{\varphi}_i) E\left\{({1\over i}\sum_{j=1}^i  \xi_{i j})^\beta \right\} = m_\beta(\mu).
\end{equation}
Notice now that by strict convexity  we have  for $\beta>1$ that
$\left({1\over i} \sum\limits_{j=1}^i \xi_{i j}\right)^\beta < {1\over i} \sum\limits_{j=1}^i \xi_{i j}^\beta$
with equality only when $\xi_{i j}=\xi_{i j'}$ for $j'\in \left\{1,2,\ldots,i\right\}$.

Therefore when there are at least one element in $\Lambda$ before $i\in \Lambda$ and $\tilde{\varphi}_j=\delta_1$ for $j \in \Lambda$ we have
\begin{equation}\label{imineq1}
E\left\{({1\over i}\sum_{j=1}^i \tilde{\eta}_i\xi_{i j})^\beta \right\} < \sum_{j=1}^i {1\over i}E\left\{\tilde{\eta}_i^\beta\xi_{ij}^\beta\right\}=m_\beta(\tilde{\eta}_i) m_\beta(\mu)=m_\beta(\mu)
\end{equation}
with strict inequality whenever $m_\beta(\mu)<\infty$. Since $\xi_{i j}$ and $\xi_{i j'}$ are independent for $j'\neq j$, they should be deterministic and because $m_1(\mu)=1$ we have that $\mu=\delta_1$. When $\mu=\delta_1$ by
 \eqref{imineq0} and \eqref{imineq1} we have
 \begin{equation}
 E\left\{\tilde{\eta}_\tau^\beta({1\over \tau}\sum_{j=1}^\tau \xi_{\tau j})^\beta\right\}=m_\beta(\tilde{\varphi}_i) m_\beta(\mu)=m_\beta(\mu)
 \end{equation}
 which implies that $m_\beta(\tilde{\varphi}_i)=1=m_1(\tilde{\varphi}_i)$, which is again true only when  $\tilde{\varphi}_i=\delta_1$.
\end{proof}
In the next Proposition we adopt some arguments of Proposition 2.1 of \cite{lastr}
\begin{propos}\label{supp}
Assume that for $1\neq k\in \Lambda$ there are $q_1, q_2\in supp \varphi_k$ such that $q_1>{1\over k}$, $q_2<{1\over k}$. Under the above assumptions, when $\mu^*$ is a fixed point of $\bP$, its support is either $\left\{0\right\}$ or $[0,\infty)$.
\end{propos}
\begin{proof}
Notice that the support of each $\varphi_i$ contains an element not greater that ${1\over i}$. Therefore taking into account that $q_2<{1\over k}$ we have that when $0\neq r\in supp \mu^*$ then the support of $\bP \mu^*$ contains an element not greater than  $$\sum\limits_{i=1, i\neq k}\limits^\infty \alpha_i r+\alpha_k q_2kr=\sum\limits_{i=1}\limits^\infty \alpha_i r+\alpha_k(q_2-{1\over k})kr=r(1-(1-kq_2)\alpha_k)).$$ Since  $(1-(1-kq_2)\alpha_k)<1$ iterating the operator $\bP$ we obtain that $0\in supp \mu^*$. Therefore to show that $supp \mu^*$ in the case when $supp \mu^*\neq \left\{0\right\}$ covers whole interval $[0,\infty)$ it suffices to show that for any $r\in (0,\infty)\cap supp \mu^*$ we have that $A:=supp \left\{\bP_k^i\mu^*, \ for \ i=1,2,\ldots\right\}$ is dense in $[0,r]$.
For a given $\varepsilon>0$ we can find positive integer $m$ such that $\left({1\over k}\right)^m\leq
 {\varepsilon \over r k q_1}$ and $(kq_2)^m\leq 1$. Then we can find positive integer $n$ such that $(kq_1)^{n-1}(kq_2)^m\leq ({1 \over k})^mkq _1\leq 1$ and $(kq_1)^{n}(kq_2)^m > 1$. Consequently we have that
$q_2^m (kq_1)^n\leq {\varepsilon\over r}$ and $k^m q_2^m (kq_1)^n \geq 1$. Therefore $iq_2^m (kq_1)^nr\in A$ for $i=1,2,\ldots,k^m$, and $q_2^m (kq_1)^nr\leq \varepsilon$ while $k^mq_2^m (kq_1)^nr\geq r$. Since $\varepsilon$ can be chosen arbitrarily small and $A$ is closed we have that $[0,\infty)\subset A$, which we obtain taking into account that together with $r$ the values $kq_1^i$ are in $supp \mu^*$.
\end{proof}
\begin{rem}
In some cases there is only one fixed point of $\bP$ which is $\delta_0\notin D$. Given probability measure $\mu$ on  $\mathbb R_{+}$ consider its characteristic function $\psi(t)=\int e^{itx}\mu(dx)$. If $\mu$ is a fixed point of $\bP$ then in the case when $\Lambda=\left\{2\right\}$,  $\psi$ satisfies the following function equation
\begin{equation}\label{eqqq}
\psi(t)=\int\limits_{\mathbb R_{+}} \left[\psi(tz)\right]^2\varphi_2(dz)
\end{equation}
Assume $\varphi_2={1\over 2} \delta_0 + {1\over 2} \delta_1$. Then \eqref{eqqq} takes the form
\begin{equation}
\psi(t)={1\over 2} + {1\over 2} \left[\psi(t)\right]^2
\end{equation}
The solutions to this quadratic equation are constant functions $\psi(t)$ and since $\psi(0)=1$ the only solution which is a characteristic function is $\psi(t)\equiv 1$, which corresponds to $\mu=\delta_0$. Consequently we don't have fixed point of $\bP$ in the set $D$.
\end{rem}

In the space $ {\mathcal M}_1(\mathbb R_{+}) $ consider   {\it Kantorovich - Wasserstein metric}  (see \cite{1 hutchinson}, Definition 4.3.1 or \cite{Bogachev}) given by the
formula
\begin{equation}\label{subsec2:1.2}
 \| \mu_1 - \mu_2 \|_{\mathcal
K} = \sup \{ | \mu_1(f) - \mu_2(f) | : \ f \in
{\mathcal K} \} \qquad\text{for}\qquad \mu_1, \mu_2\in {\mathcal
M}_{1}(\mathbb R_{+}),
 \end{equation}
 where ${\mathcal K}$ is the set of functions $
f: \mathbb R_{+} \to R $ which satisfy the condition
\begin{equation*}
 \ | f(x) - f(y) | \le |x-y| \quad \text{\rm for} \quad x, y \in \mathbb R_{+} .
\end{equation*}
We recall now another, Forter-Mourier metric $\|\cdot \|_{\mathcal F}$ in $ {\mathcal M}_1(\mathbb R_{+}) $
\begin{equation}
 \| \mu_1 - \mu_2 \|_{\mathcal
F} = \sup \{ | \mu_1(f) - \mu_2(f) | : \ f \in
{\mathcal F} \} \qquad\text{for}\qquad \mu_1, \mu_2\in {\mathcal
M}_{1}(\mathbb R_{+}),
\end{equation}
where ${\mathcal F}$ consists of functions $f$ from ${\mathcal K}$ such that $\|f\|_{sup}=\sup\limits_{x\in \mathbb R_{+}}|f(x)|\leq 1$.

Almost immediately we have that  $\| f \|_{\mathcal F}\leq \| f \|_{\mathcal K}$ for $\mu\in {\mathcal M}_{1}(\mathbb R_{+})$ which means that Kantorovich - Wasserstein metric  is stronger than Fortet Mourier metric. Notice however that  the convergence of probability measures in Fortet Mourier metric is equivalent to weak convergence  (see \cite{ethier kurtz} or \cite{Bogachev}). We have

\begin{lem}
The operator $\bP$ is nonexpansive in $\bar{D}$ in Kantorovich - Wasserstein metric.
\end{lem}
\begin{proof}
We are going to show that for each $i=1,2,\ldots$ the operators $\bP_i$ are nonexpansive in Kantorovich - Wasserstein metric.
In fact, for $f\in {\mathcal K}$ we and $\mu\in \bar{D}$ we have
\begin{equation}
\bP_i\mu(f)=\int\limits_{\mathbb R_{+}} \ldots \int\limits_{\mathbb R_{+}} f(r(x_1+x_2+\ldots+x_i))\varphi_i(dr)\mu(dx_1)\mu(dx_2)\ldots \mu(dx_i).
\end{equation}
Define for $\nu\in \bar{D}$
\begin{eqnarray}
&&\tilde{f}(x):=\int\limits_{\mathbb R_{+}} \ldots \int\limits_{\mathbb R_{+}}f(r(x+x_2+\ldots+x_i)){\varphi}_i(dr)(\mu(dx_2)\ldots \mu(dx_i)+ \nonumber \\
&& \nu(dx_2)\mu(dx_3)\ldots \mu(dx_i)+\nu(dx_2)\nu(dx_3)\mu(dx_4)\ldots \mu(dx_i)+\ldots + \nonumber \\
&& \nu(dx_2)\nu(dx_3)\ldots \nu(dx_i)).
\end{eqnarray}
Then $\tilde{f}\in {\mathcal K}$ (since $m_1({\varphi_i})={1\over i}$).
Furthermore
\begin{equation}
\bP_i\mu(f)-\bP_i\nu(f)=\mu(\tilde{f})-\nu(\tilde{f}).
\end{equation}
Hence $\|\bP_i\mu - \bP_i\nu \|_{\mathcal K}\leq \|\mu-\nu\|_{\mathcal K}$ and consequently
\begin{equation}
\|\bP\mu-\bP\nu\|_{\mathcal K}\leq \sum_{i=1}^\infty \alpha^i \|\bP_i\mu - \bP_i\nu \|_{\mathcal K}\leq \|\mu-\nu\|_{\mathcal K},
\end{equation}
which completes the proof.
\end{proof}
\begin{rem}
One can notice that operator $\bP$ is not nonexpansive in Fortet Mourier metric.
\end{rem}
We also have
\begin{coll}\label{cornon}
Operator $\bP$ transforms $D$ into itself and is defined also as a limit of $\sum\limits_{i=1}\limits^{n} \alpha_i \bP_i$  in  Kantorovich - Wasserstein metric. Furthermore $D$ is convex and closed in Kantorovich - Wasserstein metric.
\end{coll}
\begin{proof} Notice that for $\mu\in D$ we have $\left(\sum\limits_{i=1}\limits^n \alpha_i\right)^{-1} \sum\limits_{i=1}\limits^n \alpha_i \bP_i\mu\Rightarrow \bP\mu$ as $n\to \infty$ and $$m_1(\left(\sum_{i=1}^n \alpha_i\right)^{-1} \sum_{i=1}^n \alpha_i \bP_i\mu)=1=m_1(\bP\mu).$$ Therefore by the second part of Theorem  \ref{unif} we have that $\bP\mu$ is a limit of $\sum\limits_{i=1}\limits^n \alpha_i \bP_i\mu$ in  Kantorovich - Wasserstein metric. Convexity of $D$ is obvious. Closedness of $D$ in Kantorovich - Wasserstein metric follows almost immediately from the following arguments: convergence in Kantorovich - Wasserstein metric implies weak convergence and therefore the limit is a probability measure. Furthermore convergence in  Kantorovich - Wasserstein metric implies convergence of the first moments.
\end{proof}

\section{Fixed point of $\bP$ using Zolotariev seminorm}

In this section we shall introduce Zolotariev seminorm. Namely for  $\mu\in \mathcal{M}_{sig}(\mathbb{R}_{+})$ and $r\in (1,2)$ define
\begin{equation}
\|\mu\|_r:=\sup\left\{\mu(f): f\in \mathcal{F}_r\right\}
\end{equation}
 where $\mathcal{F}_r$ consists of differentiable functions $f: \mathbb R_{+} \to R $, which satisfy the condition
\begin{equation*}
 \ | f'(x) - f'(y) | \le |x-y|^{r-1} \quad \text{\rm for} \quad x, y \in \mathbb R_{+} .
\end{equation*}
One can notice that when $f\in \mathcal{F}_r$ then also function $f(x)+\alpha+\beta x$ is in $\mathcal{F}_r$.
Therefore $\|\mu\|_r=\infty$ whenever $\mu(\mathbb{R}_{+})\neq0$ or $m_1(\mu)\neq 0$. The following properties of Zolotariev seminorm will be used later on
\begin{lem}
Assume that
\begin{equation}\label{ass1}
\sum_{i=1}^\infty \alpha_i m_r(\phi_i) i^r<\infty
\end{equation}
then for $\mu, \nu\in D$, whenever $m_r(\mu)<\infty$ and $m_r(\nu)<\infty$ for $i=1,2,\ldots$, we have
\begin{equation}\label{prop1}
{1\over r}|m_r(\mu)-m_r(\nu)| \leq \|\mu-\nu\|_r\leq {1 \over r} |m_r|(\mu-\nu):={1\over r}\int\limits_{\mathbb{R}_{+}} x^r |\mu-\nu|(dx),
\end{equation}
\begin{equation}\label{prop2}
\|P_{*_i}\mu-P_{*_i}\nu\|_r\leq i \|\mu-\nu\|_r,
\end{equation}
\begin{equation}\label{prop3}
\| P_{\varphi_i}P_{*_i}\mu- P_{\varphi_i} P_{*_i}\nu\|_r\leq m_r(\varphi_i) i \|\mu-\nu\|_r,
\end{equation}
\begin{equation}\label{prop4}
\|\bP\mu-\bP\nu\|_r\leq \sum_{i=1}^\infty \alpha_i m_r(\phi_i) i \|\mu-\nu\|_r.
\end{equation}
\end{lem}
\begin{proof}
When $f\in \mathcal{F}_r$ then $f(x)=f(0)+\int_0^1 x f'(tx)dt$. Therefore
\begin{eqnarray}
 \mu(f)-\nu(f) &=& \int\limits_{\mathbb{R}_{+}}\int\limits_0\limits^1 x f'(tx)\mu(dx)-\int\limits_{\mathbb{R}_{+}}\int\limits_0\limits^1 x f'(tx)dt\nu(dx) \nonumber \\ &=& \int\limits_{\mathbb{R}_{+}}\int\limits_0\limits^1 x(f'(tx)-f'(0))dt\mu(dx)-\int_{\mathbb{R}_{+}}\int_0^1 x (f'(tx)-f'(0))dt\nu(dx) \nonumber \\
&\leq &  \int\limits_{\mathbb{R}_{+}}\int\limits_0\limits^1 x|f'(tx)-f'(0)|dt |\mu(dx)-\nu(dx)| \nonumber \\
&=&\int\limits_{\mathbb{R}_{+}}\int\limits_0\limits^1 x |t x|^{r-1}dt |\mu(dx)-\nu(dx)|={1\over r} |m_r|(\mu-\nu),
\end{eqnarray}
which completes the proof of the second part of \eqref{prop1}.
For $f(x)={1\over r}x^r$ we have that $f\in \mathcal{F}_r$ and then
\begin{equation}
\|\mu-\nu\|_r\geq {1\over r} |\int\limits_{\mathbb{R}_{+}} x^r\mu(dx) - \int\limits_{\mathbb{R}_{+}} x^r \nu(dx)|\geq {1\over r} |m_r(\mu)-m_r(\nu)|.
\end{equation}
Therefore we have \eqref{prop1}.

To prove \eqref{prop2} notice that $P_{*_i}\mu-P_{*_i}\nu=\sum\limits_{j=1}\limits^i \mu_j *\mu -  \sum\limits_{j=1}\limits^i \mu_j *\nu$
where  $\mu_j=(P_{*_{i-j}}\mu)*(P_{*_{j-1}}\nu)$ with $P_{*_0}\mu=P_{*_0}\nu=\delta_0$. When $f\in\mathcal{F}_r$ then  $\bar{f}(y)=\int
\limits_{\mathbb{R}_{+}}f(x+y)\mu_j(dx)$ is in $\mathcal{F}_r$. Therefore
\begin{eqnarray}
\|\mu_j*(\mu-\nu)\|_r&=&\sup_{f\in {\mathcal{F}_r}}|\int\limits_{\mathbb{R}_{+}}\int\limits_{\mathbb{R}_{+}} f((x+y)\mu_j(dx)(\mu-\nu)(dy)| \nonumber \\
&=&\sup_{f\in {\mathcal{F}_r}}|\int\limits_{\mathbb{R}_{+}} \bar{f}(y)(\mu-\nu)(dy)|\leq \|\mu-\nu\|_r,
\end{eqnarray}
from which \eqref{prop2} easily follows.
Now
\begin{eqnarray}
\| P_{\varphi_i}P_{*_i}\mu- P_{\varphi_i} P_{*_i}\nu\|_r &=&\sup_{f\in {\mathcal{F}_r}} \left(\int\limits_{\mathbb{R}_{+}}\int\limits_{\mathbb{R}_{+}} f(zx)\varphi_i(dz)P_{*_i}\mu(dx)- \int\limits_{\mathbb{R}_{+}}\int\limits_{\mathbb{R}_{+}} f(zx)\varphi_i(dz)P_{*_i}\nu(dx)\right) \nonumber \\
&\leq &
 m_r(\varphi_i)\sup_{f\in {\mathcal{F}_r}} \left( \int\limits_{\mathbb{R}_{+}} \tilde{f}(x)P_{*_i}\mu(dx)- \int\limits_{\mathbb{R}_{+}} \tilde{f}(x)P_{*_i}\nu(dx)\right) \nonumber \\
&\leq&  m_r(\varphi_i) \|P_{*_i}\mu-P_{*_i}\nu\|_r,
\end{eqnarray}
where $\tilde{f}(x)={1\over m_r(\varphi_i)}\int\limits_{\mathbb{R}_{+}}f(zx)\varphi_i(dx)$ is an element of $\mathcal{F}_r$,  provided that $f\in \mathcal{F}_r$. The estimation \eqref{prop3} follows now directly from \eqref{prop2}. From  \eqref{prop3} we  immediately obtain \eqref{prop4}.
\end{proof}

The main result of this section is the existence of a fixed point of $\bP$.  We use the same assumptions as in the paper \cite{20 lasota traple}, where similar result was obtained. Our proof is different and shows another properties of the fixed point of $\bP$. It is formulated as follows

\begin{thm}\label{fixedpoint}
Assume that for some $r\in (1,2)$ we have $m_r(\varphi_i)< {1\over i}$ for all $i \in \Lambda$, \eqref{ass1} is satisfied and there is $\mu \in D$ such that $m_r(\mu)<\infty$.
Then $\bP^n\mu$ converges in  Kantorovich - Wasserstein metric to a unique $\mu^*\in D$, which is a fixed point of $\bP$ such that  $m_r(\mu^*)<\infty$.
\end{thm}
\begin{proof} Under \eqref{ass1} using Proposition \ref{P1} we have that $m_r(\bP^n\mu)<\infty$ for $n=1,2,\ldots$.
By \eqref{prop4} we have that $\|\bP^{j+1}\mu-\bP^j\mu\|_r\leq \lambda^j |\bP\mu-\mu\|_r$, where $\lambda=\sum\limits_{i=1}\limits^\infty \alpha_i m_r(\phi_i) i$. Therefore
\begin{equation}
\|\bP^n\mu-\mu\|_r=\|\sum\limits_{j=1}\limits^n \bP^{j}\mu-\bP^{j-1}\mu\|_r\leq \sum\limits_{j=1}\limits^n \lambda^j \|\bP\mu-\mu\|_r,
\end{equation}
where $\bP^0\mu=\mu$. Since by \eqref{prop1}
\begin{equation}
\|\bP^n\mu-\mu\|_r\geq {1\over r}|m_r(\bP^n\mu)-m_r(\mu)|,
\end{equation}
we therefore have that for $n=1,2,\ldots$
\begin{equation}\label{rmom}
m_r(\bP^n\mu)\leq {r\over 1-\lambda}\|\bP\mu-\mu\|_r+m_r(\mu):=\kappa < \infty.
\end{equation}
The sequence of probability measures $\left\{\mu,\bP\mu,\bP^2\mu,\ldots\right\}$ is tight and therefore there is $\mu^*$ and subsequence $(n_k)$ such that $\bP^{n_k}\mu\Rightarrow \mu^*$ as $k\to \infty$. By \eqref{rmom} and Corollary \ref{impcor}  we have that $\|\bP^{n_k}\mu -\mu^*\|_{\mathcal{K}}\to 0$ as $k\to \infty$. Consequently $\mu^*\in D$.  Since for $K>0$
\begin{equation}
\int\limits_{\mathbb{R}_{+}}(x^r\wedge K)\bP^{n_k}\mu(dx)\leq m_r(\bP^{n_k}\mu)\leq \kappa
\end{equation}
and therefore letting $k\to \infty$ we have that $\int_{\mathbb{R}_{+}}(x^r\wedge K)\mu^*(dx)\leq \kappa$, which by Fatou lemma gives that $m_r(\mu^*)\leq \kappa$.
Now
\begin{equation}\label{rozw1}
\|\bP^{n_{k}}\bP\mu-\bP^{n_{k}}\mu\|_{\mathcal{K}}=\|\bP\bP^{n_{k}}\mu-\bP^{n_{k}}\mu\|_{\mathcal{K}}\to \|\bP\mu^*-\mu^*\|_{\mathcal{K}}:=\beta
\end{equation}
and
\begin{eqnarray}\label{rozw2}
&&\|\bP \bP \mu^* - \bP\mu^*\|_{\mathcal{K}}=\lim_{k\to \infty} \|\bP \bP \bP^{n_{k}}\mu - \bP\bP^{n_{k}}\mu\|_{\mathcal{K}}\geq \nonumber \\
&&\lim_{k\to \infty} \|\bP \bP^{n_{k+1}}\mu - \bP^{n_{k+1}}\mu\|_{\mathcal{K}}=\|\bP\mu^*-\mu^*\|_{\mathcal{K}}=\beta,
\end{eqnarray}
so that taking into account that $\|\bP \bP \mu^* - \bP\mu^*\|_{\mathcal{K}}\leq \|\bP\mu^*-\mu^*\|_{\mathcal{K}}$ we obtain that
$\|\bP \bP \mu^* - \bP\mu^*\|_{\mathcal{K}} = \|\bP\mu^*-\mu^*\|_{\mathcal{K}}$.
By small modification of \eqref{rozw1} and \eqref{rozw2} we obtain that for any $n=0,1,\ldots$
\begin{equation}
\|\bP^{n+1} \mu^* - \bP^n\mu^*\|_{\mathcal{K}} = \|\bP\mu^*-\mu^*\|_{\mathcal{K}}.
\end{equation}
On the other hand by \eqref{prop4} we have that $\|\bP^{n+1} \mu^* - \bP^n\mu^*\|_r\leq \lambda^n \|\bP\mu^*-\mu^*\|_r$.
Therefore $\lim\limits_{n\to \infty}\|\bP^{n+1} \mu^* - \bP^n\mu^*\|_r=0$. By Theorem \ref{Rio} we also have that $\lim\limits_{n\to \infty}\|\bP^{n+1} \mu^* - \bP^n\mu^*\|_{\mathcal{K}}=0$. Therefore $\|\bP\mu^*-\mu^*\|_{\mathcal{K}}=0$ and $\mu^*$ is a fixed point of $\bP$. If there is another weak limit $\nu^*$ of subsequence of $\bP^n\mu$, then $\|\bP^n\mu^*-\bP^n\nu^*\|_r\to 0$, as $n\to \infty$ and by Theorem \ref{Rio} again we have that $\|\mu^*-\nu^*\|_{\mathcal{K}}=\lim\limits_{n\to \infty}\|\bP^{n+1} \mu^* - \bP^n\mu^*\|_{\mathcal{K}}=0$.
Consequently any weak limit of a subsequence of  $\bP^n\mu$ is equal to $\mu^*$, which means that $\bP^n\mu\Rightarrow \mu^*$ and the convergence also holds in Kantorovich - Wasserstein metric.
\end{proof}
\begin{rem}
We may not exclude the case in which we have another fixed point $\nu^*\in D$ of $\bP$ such that $m_r(\nu^*)=\infty$ for each $r>1$.
Notice furthermore that in the case when $\varphi_i$ is uniformly distributed over the interval $[0,{1\over i}]$ then we have  $m_r(\varphi_i)={1\over r+1} \left({2\over i}\right)^r<{1\over i}$ for all $r\in (1,2)$ and $i\in \Lambda\setminus\left\{1\right\}$, so that if \eqref{ass1} is satisfied and $1\notin \Lambda$, by the above theorem we have existence of a unique fixed point of $\bP$ in $D$ with finite $r$-th moment.
\end{rem}

\section{Contraction property of the operator $\bP$}

We have the following generalization of Theorem 5.2.2 of \cite{7}
\begin{thm}\label{thm71}
Assume for some $i\in \Lambda$ we have that $0$  is an accumulation point of $\varphi_i$, where $m_1(\varphi_i)={1\over i}$.
Then for $\mu,\nu\in {D}$  such that $\mu\neq \nu$ and
\begin{equation}\label{equ75}
supp(P_{*(i-1)}(\mu+\nu))=\mathbb{R}_+
\end{equation}

we have
\begin{equation}
\|\bP_i \mu-\bP_i \nu\|_{\mathcal{K}} < \|\mu-\nu\|_{{\mathcal{K}}}
\end{equation}
and consequently
\begin{equation}
\|\bP \mu-\bP \nu\|_{\mathcal{K}} < \|\mu-\nu\|_{\mathcal{K}}.
\end{equation}
\end{thm}
\begin{proof}
 Recall that $\bP_i={P}_{\varphi_i} {P}_{*_i}$ and assume that $\|\bP_i \mu-\bP_i \nu\|_{\mathcal{K}} = \|\mu-\nu\|_{\mathcal{K}}$. By Theorem \ref{A2} there is $f_0\in \mathcal{K}$ such that
\begin{equation}\label{=}
\|\bP_i \mu-\bP_i \nu\|_{\mathcal{K}}=\langle f_0,\bP_i \mu-\bP_i \nu\rangle.
\end{equation}
Then
\begin{eqnarray}
&&\|\mu-\nu\|_{{\mathcal{K}}}=\int_{{\mathbb{R}_+}^{i+1}}f_0((x_1+x_2+\ldots+x_i)r){\varphi}_i(dr) \nonumber \\
&& \left[\mu(dx_1)\mu(dx_2)\ldots \mu(dx_i)-\nu(dx_1)\nu(dx_2)\ldots \nu(dx_i)\right]=\langle f_1,\mu-\nu \rangle,
\end{eqnarray}
where
\begin{eqnarray}
&&f_1(x)=\int_{{\mathbb{R}_+}^{i}}f_0((x+x_2+\ldots+x_i)r){\varphi}_i(dr) \left[\mu(dx_2)\ldots\mu(dx_i)+ \right. \nonumber \\
&& \nu(dx_2)\mu(dx_3)\ldots\mu(dx_i)+\nu(dx_2)\nu(dx_3)\mu(dx_4)\ldots \mu(dx_i)+ \ldots + \nonumber \\
&&\left. \nu(dx_2)\nu(dx_3)\nu(dx_4)\ldots \mu(dx_i) + \nu(dx_2)\nu(dx_3)\nu(dx_4)\ldots \nu(dx_i)\right].
\end{eqnarray}
Clearly, using again the fact that $m_1(\varphi_i)={1\over i}$ we have that ${f}_1\in \mathcal{K}$.

By Theorem \ref{A2} there are two points $x_1, x_2 \in \mathbb{R}_+$ such that $x_1<x_2$ and $|{f}_1(x_2)-{f}_1(x_1)|=x_2-x_1$. Since ${f}_1$ is nonexpansive (Lipschitz with constant less or equal to $1$) we have that ${f}_1(x)=\theta x+\sigma$, for $x\in (x_1,x_2)$ with $|\theta|=1$. Therefore $|{f}_1(x_1+\varepsilon)-{f}_1(x_1)|=\varepsilon$ for $\varepsilon\in(0,x_2-x_1)$. Replacing $f_0$ by $-f_0$ we may assume that ${f}_1(x_1+\varepsilon)-{f}_1(x_1)=\varepsilon$ for $\varepsilon\in(0,x_2-x_1)$. We are going now to show that for $x\in \mathbb{R}_+$
\begin{equation}\label{id}
f_0(x)=x+c
\end{equation}
with a constant $c\in \mathbb{R}$.
Consider now $u_1,u_2 \in \mathbb{R}_+$ such that $u_1<u_2$. We want to show that then
\begin{equation}\label{contr1}
f_0(u_2)-f_0(u_1)\geq u_2-u_1,
\end{equation}
which by nonexpansiveness of $f_0$ implies that $f_0(u_2)-f_0(u_1)=u_2-u_1$ and therefore $f_0$ is of the form \eqref{id}. Assume conversely that $f_0(u_2)-f_0(u_1)<u_2-u_1$. Since $f_0$ as a Lipschitzian mapping is almost everywhere differentiable there is $\bar{u}\in (u_1,u_2)$ such that $f_0'(\bar{u})<1$ and for $\delta\in (0,\delta_0)$ we have
\begin{equation}
{f_0(\bar{u}+\delta)-f_0(\bar{u}) \over \delta} <1.
\end{equation}
Define
\begin{equation}\label{equal0}
h(y_2,\ldots,y_i,r,\varepsilon)={f_0((x_1+\varepsilon+y_2+\ldots+y_i)r)-f_0((x_1+y_2+\ldots+y_i)r) \over \varepsilon r}.
\end{equation}
By definition of $f_1$ for $\varepsilon\in (0,x_2-x_1)$ we have
\begin{eqnarray}\label{equal1}
&& 1={{f}_1(x_1+\varepsilon) - {f}_1(x_1) \over \ve}= \int_{(\mathbb{R}_+)^{i-1}} \int_{\mathbb{R}_+} h(y_2,\ldots,y_i,r,\ve)r {\phi}_i(dr) \left[\mu(dy_2) \right. \nonumber \\
&&\ldots\mu(dy_i)+
\nu(dy_2)\mu(dy_3)\ldots\mu(dy_i)+\nu(dy_2)\nu(dy_3)\mu(dy_4)\ldots \mu(dy_i)+   \nonumber \\
&& \left. \ldots +\nu(dy_2)\nu(dy_3)\nu(dy_4)\ldots \mu(dy_i) + \nu(dy_2)\nu(dy_3)\nu(dy_4)\ldots \nu(dy_i)\right]:= \nonumber \\
&&\int_{(\mathbb{R}_+)^{i}} h(y_2,\ldots,y_i,r,\ve)q(dy_2,\ldots,dy_i,dr),
\end{eqnarray}
where we define implicitly probability measure $q$. Since $0$ is an accumulation point of $supp \phi$ there is $\bar{r}\in supp {\phi}$ such that $x_1 \bar{r} < \bar{u}$. Then there is $(\bar{y}_2,\bar{y}_3,\ldots,\bar{y}_i)\in supp(P_{*(i-1)}(\mu+\nu))$ such that
\begin{equation}\label{equal2}
\bar{u}-x_1\bar{r}=(\bar{y}_2+\bar{y}_3+\ldots,\bar{y}_i)\bar{r}.
\end{equation}
Consequently for every $\bar{\ve}\in (0,x_2-x_1)$ such that $\bar{\ve}\bar{r}<\delta_0$ we have
\begin{equation}
h(\bar{y}_2+\bar{y}_3+\ldots,\bar{y}_i,\bar{r},\bar{\ve})={f_0(\bar{u}+\bar{\ve} \bar{r})-f_0(\bar{u}) \over \bar{\ve}\bar{r}}<1.
\end{equation}
Since $h\leq 1$ by continuity of $h$ and full support of $P_{*(i-1)}(\mu+\nu)$ we have that
\begin{equation}
\int_{(\mathbb{R}_+)^{i}} h(y_2,\ldots,y_i,r,\ve)q(dy_2,\ldots,dy_i,dr)<1,
\end{equation}
a contradiction to  \eqref{equal1}. Therefore we have equality in \eqref{contr1}. Consequently $f_0(x)=x+c$ for a constant $c$. Since
$\bP_i\mu$ and $\bP_i\nu\in D$ we therefore have $\langle f_0,\bP_i\mu-\bP_i\nu\rangle=m_1(\bP_i\mu)-m_1(\bP_i\nu)=0$ and by \eqref{=} $$\|\bP_i\mu-\bP_i\nu\|_{\mathcal{K}}=\|\mu-\nu\|_{\mathcal{K}}=0,$$
 which contradicts the fact that $\mu\neq \nu$.

\end{proof}

\begin{rem}
Condition $supp(P_{*(i-1)}(\mu+\nu))=\mathbb{R}_+$ is not very restrictive. It holds in particular when $supp(\mu+\nu)=\mathbb{R}_+$ or
when $supp \mu=\mathbb{R}_+$.
\end{rem}

We have the following consequences of Theorem \ref{thm71}

\begin{coll}\label{c1}
If for some $i \in \Lambda$ we have that $0$ is an accumulation point of $\phi_i$ and $\mu^*$ is a weak accumulation point of $\bP^n\mu$ for $\mu \in {D}$  i.e. there is a sequence $(n_k)$ such that $\bP^{n_k}\mu\Rightarrow \mu^*$, as $k\to \infty$ and  $supp \mu^*=\mathbb{R}_+$, $\mu^*\in D$  then $\mu^*$ is a fixed point of $\bP$.
\end{coll}
\begin{proof}
Notice that $m_1(\bP^{n_k}\mu)=m_1(\mu^*)$ so that by Theorem \ref{unif} we have that $\|\bP^{n_k}\mu-\mu^*\|_{\mathcal{K}}\to 0$ as $k\to \infty$. Therefore
\begin{equation}
\|\bP^{n_{k+1}}\bP\mu-\bP^{n_{k+1}}\mu\|_{\mathcal{K}}\leq \|\bP^{n_{k}+1}\bP\mu-\bP^{n_{k}+1}\mu\|_{\mathcal{K}}\leq \|\bP^{n_{k}}\bP\mu-\bP^{n_{k}}\mu\|_{\mathcal{K}} \leq \|\bP\mu-\mu\|_{\mathcal{K}},
\end{equation}
so that there is a limit $\lim_{k\to \infty} \|\bP^{n_{k}}\bP\mu-\bP^{n_{k}}\mu\|_{\mathcal{F}_0}:=\beta$ and by continuity
\begin{equation}
\|\bP^{n_{k}}\bP\mu-\bP^{n_{k}}\mu\|_{\mathcal{K}}=\|\bP\bP^{n_{k}}\mu-\bP^{n_{k}}\mu\|_{\mathcal{K}}\to \|\bP\mu^*-\mu^*\|_{\mathcal{K}}=\beta.
\end{equation}
If $\bP\mu^*\neq \mu^*$ and assumptions of Corollary are satisfied then by Theorem \ref{thm71} we have that
$\|\bP \bP \mu^* - \bP\mu^*\|_{\mathcal{K}}<\|\bP\mu^*-\mu^*\|_{\mathcal{K}}=\beta$, while
\begin{eqnarray}
&&\|\bP \bP \mu^* - \bP\mu^*\|_{\mathcal{K}}=\lim_{k\to \infty} \|\bP \bP \bP^{n_{k}}\mu - \bP\bP^{n_{k}}\mu\|_{\mathcal{K}}\geq \nonumber \\
&&\lim_{k\to \infty} \|\bP \bP^{n_{k+1}}\mu - \bP^{n_{k+1}}\mu\|_{\mathcal{K}}=\|\bP\mu^*-\mu^*\|_{\mathcal{K}}=\beta
\end{eqnarray}
and we have a contradiction. Therefore $\bP\mu^* = \mu^*$.
\end{proof}

\begin{coll}\label{c2}
If for some $i \in \Lambda$ we have that $0$ is an accumulation point of $\phi_i$ and $\mu^*\in {D}$ is a fixed point of $\bP$,  then there are no other fixed point $\nu^*\in {D}$. Consequently for $\nu\in D$ we have that any weakly convergent subsequence $\bP^{n_k} \nu$ either converges to $\mu^*$, as $n\to \infty$, or to a measure $\tilde{\mu}\in \bar{D}\setminus D$.
\end{coll}
\begin{proof} By Proposition \ref{supp} since  $\mu^*\in {D}$ we have that
$supp \mu^*= \mathbb{R}_+$.
Then we use again Theorem \ref{thm71}. Namely when $\nu^*\in {D}$ is another fixed point of $\bP$  we have
\begin{equation}
\|\mu^*-\nu^*\|_{\mathcal{K}}=\|\bP\mu^*-\bP\nu^*\|_{\mathcal{K}}<\|\mu^*-\nu^*\|_{\mathcal{K}},
\end{equation}
which is a contradiction. Any sequence $(\bP^n\nu)$ is compact in Fortet Mourier metric and its subsequence converges to a measure $\tilde{\mu}$, and the convergence is also in Kantorovich Wasserstein metric whenever $m_1(\mu)=1$. In the last case we have $\tilde{\mu}=\mu^*$ by uniqueness of fixed point of $\bP$ in $D$.
\end{proof}

Using Theorem \ref{contr*} we can now strengthen the last Corollary.
\begin{coll}\label{thm72}
Assume there is $i\in \Lambda$ such that $0$  be an accumulation point of $\varphi_i$. Assume that
$\mu^*\in {D}$ is a fix point of $\bP$.  Then for $\nu\in {D}$ for any limit of weakly convergent subsequence $\bP^{n_k} \nu$ which belongs to $D$ is also in Kantorovich Wasserstein metric. Furthermore if the sequence $n\to \bP^n\nu$ is uniformly integrable we have
\begin{equation}\label{conv*}
\lim_{n\to \infty}\|\bP^n \nu-\mu^*\|_{\mathcal{K}}=0.
\end{equation}
\end{coll}
\begin{proof} As above using Proposition \ref{supp}, since  $\mu^*\in {D}$, we have that
$supp \mu^*= \mathbb{R}_+$. The first part easily follows from Corollary \ref{c2}.
 It remains to notice only that uniform integrability of $n\to \bP^n\nu$ together with weak compactness implies compactness of $n\to \bP^n\nu$ in  Kantorovich Wasserstein metric. Since any subsequence converges to the same limit in Kantorovich Wasserstein metric  we have \eqref{conv*}.
\end{proof}

\section{Asymptotic stability of the nonlinear Boltzmann-type equation }

Consider now the following equation in the space of signed measures

\begin{equation}\label{Bo1}
\frac{d\psi}{dt} + \psi = \bP\,\psi \qquad\text{for}\qquad t \geq 0
\end{equation}
with initial condition
\begin{equation}\label{Bo2}
 \psi\,(0) = \psi_{0},
\end{equation}
where $\psi_0\in \mathcal{M}_1(\mathbb{R}_{+})$ and   $ \psi : \mathbb{R}_{+} \rightarrow \mathcal{M}_{sig}(\mathbb{R}_{+})$.
In this section we show that the equation (\ref{s4.w1.11adodm2})
 may by considered in a convex closed subset $D$ of a
vector space of signed measures $\mathcal{M}_{sig}(\mathbb{R}_{+})$. This approach seems to be quite
natural and it is related to the classical results concerning the
semigroups and differential equations on convex subsets of Banach
spaces (see \cite{crandall}, \cite{20 lasota traple}). For details see Appendix.

We finish the paper with  sufficient conditions for the
asymptotic stability of solutions of the equation \eqref{Bo1} with respect to Kantorovich--Wasserstein  metric.

Equation \eqref{Bo1}  together with the initial condition \eqref{Bo2} may be considered
in a convex subset ${D}$ of the vector space of finite signed measures  $\mathcal{M}_{sig}$.

\begin{coll}\label{cola}
Assume    $\varphi_i \in \mathcal{M}_1, i\in \{1, 2, ...,\}  $ is such that  $m_1(\varphi_i)={1\over i}$, $\bP$ is given by \eqref{equ61} with $\sum\limits_{i=1}\limits^\infty \alpha_i=1$, $\alpha_i\geq 0$.  Then for every $ \psi_{0} \in D$ there exists a unique solution $\psi$ of problem \eqref{Bo1}, \eqref{Bo2} taking values in ${D}$.
\end{coll}

The solutions of \eqref{Bo1}, \eqref{Bo2}  generate a semigroup of Markov operators $(P^{t})_{t\geq 0}$
on D given by
\begin{equation}\label{s4.1a19}
    {P}^{t}\,u_{0} = u(t) \qquad\text{for}\qquad t\in {\mathbb R}_{+},\qquad u_{0}\in {D}.
\end{equation}

Now using Theorem \ref{thm71} we can easily derive the following result
\begin{thm}\label{thm72}
Let $\bP$ be operator given by \eqref{equ61}. Moreover, let
$(\varphi_1,\varphi_2,...)$ be a sequence of probability measures such that $m_1(\varphi)={1\over i}$ and   $\alpha_i\geq 0$ be a sequence of nonnegative numbers such that  $\sum\limits_{i=1}\limits^\infty \alpha_i=1$. Assume that $0$ is an accumulation point of $supp\varphi_i$ for some $i\in \Lambda$.
 If $\bP$ has a fixed point $\mu^* \in D $ then
\begin{equation}\label{s4.1a20}
\lim_{t\to\infty} ||\psi (t)- \mu^*||_{\mathcal{K}} = 0
\end{equation}
for every sequentially compact (in Kantorovich Wasserstein metric) solution $ \psi $ of  \eqref{Bo1}, \eqref{Bo2}.
\end{thm}

\begin{proof}
First we have to prove that $(P^{t})_{t\geq 0}$ is
nonexpansive on $D$ with respect to Kantorovich-Wasserstein  metric. For this purpose let $\eta_{0}, \vartheta_{0} \in D$ be given. For $t\in {\mathbb R}_{+}$ define
\begin{equation}
    \upsilon(t) = P^{t}\,\eta_{0} - P^{t}\,\vartheta_{0}.
\end{equation}

Using  \eqref{s4.1a18}, Corollary \ref{cornon}  and \eqref{s4.1a19}  it is easy to see that
\begin{equation}
v(t)=  e^{-t} v_0 + \int\limits_{0}\limits^{t} e^{-(t-s)} (\bP(P^s \eta _0) - \bP(P^s \vartheta_0) ) ds \quad\textnormal{for} \quad t\in \mathbb{R}_+.
\end{equation}
From Corollary \ref{cola}  it follows immediately that
\begin{equation}
||v(t)||_{\mathcal{K}}  \leq  e^{-t} ||v(0)||_{\mathcal{K}} + \int\limits_{0}\limits^{t} e^{-(t-s)} ||v(s)||_{\mathcal{K}} ds \quad\textnormal{for} \quad t\in \mathbb{R}_+.
\end{equation}
Consequently
\begin{equation}
f(t) \leq  ||v(0)||_{\mathcal{K}} + \int\limits_{0}\limits^{t} f(s) ds \quad\textnormal{for} \quad t\in \mathbb{R}_+,
\end{equation}
where $f(t)=e^t  ||v(t)||_{\mathcal{K}} $.
From the Gronwall inequality it follows that
\begin{equation}
f(t) \leq e^{t} ||v(0)||_{\mathcal{K}}
\end{equation}
This is equivalent to the fact that $(P^t)_{\geq t}$ is nonexpansive on $D$ with respect to Kantorovich - Wasserstein metric.
Notice that $\mu^*$ as a fixed point of $\bP$ is a stationary solution to the equation \eqref{Bo1} i.e. $P^t \mu^*=\mu^*$.
To complete the proof it is sufficient to verify condition \eqref{strict} of Theorem \ref{contr*}.

From (\ref{s4.1a18}) and Proposition \ref{supp} and  Theorem \ref{thm71} it follows immediately that for $\psi_{0}\in
D$ and  $t > 0$
\begin{eqnarray}\label{rown}
&&\|P^{t}\psi_{0} - \mu^*\|_{\mathcal{K}} \leq
e^{-t}\,\|\,\psi_{0} - \mu^*\|_{\mathcal{K}}
+ \int\limits_{0}\limits^{t}\,e^{-(t-s)}\,\|\bP P^s\psi_{0} - \bP\mu^*\|_{\mathcal{K}}\, ds < \nonumber \\
&&e^{-t}\,\|\,\psi_{0} - \mu^*\|_{\mathcal{K}} + (1 - e^{-t})\,\|P^{s}\psi_{0} - \mu^*\|_{\mathcal{K}}\leq \|\,\psi_{0}
- \mu^*\|_{\mathcal{K}}.
\end{eqnarray}

By Theorem \ref{contr*} we immediately obtain \eqref{s4.1a20}.
\end{proof}

We shall now study nonlinear Boltzmann equation \eqref{Bo1} using Zolotariev seminorm following the results of \cite{20 lasota traple}.
Consider time discretized version of \eqref{Bo1} with discretization step $h\in (0,1)$
\begin{equation}\label{Bo3}
\frac{d\psi_{h}}{dt}(d_h(t)) + \psi_{h}(d_h(t)) = \bP\,\psi_{h} (d_h(t)) \qquad\text{for}\qquad t \geq 0
\end{equation}
with initial condition
\begin{equation}\label{Bo4}
 \psi_h\,(0) = \psi_{0},
\end{equation}
where $\psi_0\in \mathcal{M}_1(\mathbb{R}_{+})$ and $d_h(t)=nh$ for $t\in [nh,(n+1)h)$.
Then
\begin{equation}
\psi_h((n+1)h)=(1-h)\psi_h(nh)+h \bP\psi_h(nh):=\bP_h\psi_h(nh).
\end{equation}
Notice that fixed point of the operator $\bP$ is also a fixed point of $\bP_h$ and vice versa.
We have
\begin{lem} Under assumptions of Theorem \ref{fixedpoint} we have that
\begin{itemize}
\item[(i)] {when $m_r(\mu)<\infty$ then $m_r(\bP_h^n \mu)<\infty$ for any positive integer $n$ and $\mu\in \mathcal{M}_{1}(\mathbb R_{+})$},
\item[(ii)] {$\|\bP_h \mu - \bP_h \nu\|_r\leq \lambda_h \|mu-\nu\|_r$ for $\mu,\nu\in \mathcal{M}_{1}(\mathbb R_{+})$ such that $m_r(\mu)<\infty$, $m_r(\nu)<\infty$ with $\lambda_h=1-h(1-\lambda)$, where $\lambda=\sum\limits_{i=1}\limits^\infty \alpha_i m_r(\varphi_i) i$},
    \item[(iii)] {$\|\bP_h^n\mu-\mu^*\|_{\mathcal{K}} \leq 2^{1+{1\over r}}\lambda_h^{{n\over r}} \left(\|\mu-\mu^*\|_r\right)^{1\over r}\leq K $ with $\mu^*\in D$ being the unique fixed point of $\bP$ and $K={1\over r} 2^{1+{1\over r}}(m_r(\mu)+m_r(\mu^*))$.}
        \end{itemize}
\end{lem}
\begin{proof}
Under \eqref{ass1} using Proposition \ref{P1} we have that $m_r(\bP\mu)<\infty$ and therefore also $m_r(\bP_h \mu)<\infty$. Then (i) follows by induction. (ii) follows directly from the definition of $\bP$ and \eqref{prop4}. For fixed point $\mu^*$ of $\bP$, which is also  a fixed point of $\bP_h$ and $\mu\in \mathcal{M}_{1}(\mathbb R_{+})$ such that $m_r(\mu)<\infty$ we have
\begin{equation}
\|\bP_h^n \mu\|_r \leq \lambda_h^n\|\mu-\nu\|_r.
\end{equation}
Therefore using Theorem \ref{Rio}  and then \eqref{prop1} we obtain
\begin{equation}
\|\bP_h^n \mu - \mu^*\|_{\mathcal{K}}\leq 2^{1+{1\over r}} \lambda_h^{n\over r} (\|\mu-\mu^*\|_{\mathcal{K}})^{1\over r}\leq K
\end{equation}
with $K={1\over r} 2^{1+{1\over r}}(m_r(\mu)+m_r(\mu^*))$.
\end{proof}
We recall now Lemma 3 of \cite{20 lasota traple}.
\begin{lem}\label{Las1}
Under the assumptions of Theorem \ref{fixedpoint}, when $\mu=\psi_0$ is such that $m_r(\mu)<\infty$, we have
\begin{equation}
\|\psi_h(t)-\psi(t)\|_{\mathcal{K}}\leq 4K h (e^{2t}-1).
\end{equation}
\end{lem}
We can now formulate
\begin{thm}
Under assumptions of Theorem \ref{fixedpoint} when $\mu=\psi_0$ is such that $m_r(\mu)<\infty$ we have that
\begin{equation}
\|\psi(t)-\mu^*\|_{\mathcal{K}}\leq K e^{-{t\over r}(1-\lambda)}.
\end{equation}
\end{thm}
\begin{proof}
We follow the arguments of the proof of Theorem 1 of \cite{20 lasota traple}.
Fix $t>0$ and for a given $\varepsilon>0$ find positive integer $n$ such that ${t\over n}4K(e^{2t}-1)\leq \varepsilon$
Then by Lemma \ref{Las1}
\begin{equation}\label{Las2}
\|\psi(t)-\mu^*\|_{\mathcal{K}}\leq \|\psi(t)-\psi_h(t)\|_{\mathcal{K}} + \|\psi_h(t)-\mu^*\|_{\mathcal{K}}\leq \varepsilon + K(1-{t\over n}(1-\lambda))^{n\over r}.
\end{equation}
Since $1-x\leq e^{-x}$ for $x\geq 0$ we have that $(1-{t\over n}(1-\lambda))^{n\over r}\leq e^{-{t\over r}(1-\lambda)}$ and the claim follows from \eqref{Las2} taking into account that $\varepsilon$ could be chosen arbitrarily small.
\end{proof}

\section{Appendix}
On a given complete metric space $(E,\rho)$ consider a continuous operator $T$ or continuous semigroup $(T_t)$, for $t\geq 0$ transforming $(E,\rho)$ into itself. Denote by $\omega(x)$ the set of all limiting points of the trajectory $n\to T^nx$ or $t\to T_tx$ respectively.   We say that $n\to T^nx$ or $t\to T_tx$ is sequentially compact if from every sequence $T^{n_k}x$, $T_{t_{n_k}}x$ respectively, one could choose a convergent subsequence. Let ${\mathcal Z}$ be the set of all $x$ such that the trajectory $t \to T_tx$ ($n\to T^nx$) is sequentially compact. We shall assume that
${\mathcal Z}$ is a nonempty set and let
$\Omega=\bigcup\limits_{\mu\in {\mathcal Z}} \omega(\mu)$.
We have the following result formulated for semigroup $T_t$, which naturally holds for continuous operator $T$

\begin{thm} \label{contr*}(see Theorem 5.1.2 of \cite{7}) Assume that $T_t$ is nonexpansive i.e.
\begin{equation}
\rho(T_tx,T_ty)\leq \rho(x,y)
\end{equation}
for $t\geq 0$ and there is $x^*\in \Omega$ such that for every $x\in \Omega$, $x\neq x^*$ there is $t(x)$ such that
\begin{equation}\label{strict}
\rho(T_{t(x)}x,T_{t(x)}x^*)<\rho(x,x^*).
\end{equation}
Then for $z\in {\mathcal Z}$ we have
\begin{equation}
\lim_{t\to \infty}\rho(T_tz,x^*)=0.
\end{equation}
\end{thm}

\begin{defin}\label{uninte}
Sequence of probability measures $\mu_n$ defined on $\mathbb{R}_+$ is uniformly integrable when
\begin{equation}\label{unint1}
\sup_n \int\limits_M\limits^\infty x \mu_n(dx) \to 0,
\end{equation}
whenever $M\to \infty$.
\end{defin}
\begin{thm}\label{unif} Assume that for sequence of probability measures $\mu_n$ defined on $\mathbb{R}_+$ we have that $m_1(\mu_n)<\infty$ and $\mu_n\Rightarrow \mu$, as $n\to \infty$. Then $m_1(\mu_n)\to m_1(\mu)$ if and only if measures $\mu_n$ are uniformly integrable. Furthermore for sequence of probability measures $\mu_n$ defined on $\mathbb{R}_+$ such that $m_1(\mu_n)<\infty$ and we have that convergence $\mu_n\Rightarrow \mu$, as $n\to \infty$, together with convergence of $m_1(\mu_n)\to m_1(\mu)$ is equivalent to convergence $\|\mu_n-\mu\|_{\mathcal{K}}\to 0$.
\end{thm}
\begin{proof}
By Skorokhod theorem (25.6 of  \cite{bil}) there is a probability space $(\Omega,F,P)$ and nonnegative random variables $X_n$, $X$ with laws $\mu_n$ and $\mu$ respectively such that $X_n(\omega)\to X(\omega)$ for each $\omega \in \Omega$. Uniform integrability of $\mu_n$ is equivalent to uniform integrability of $X_n$. By Theorem II T21 of \cite{Meyer} uniform integrability of $X_n$ is equivalent to the convergence $m_1(\mu_n)\to m_1(\mu)$. To prove the last statement of Theorem notice that when $\|\mu_n-\mu\|_{\mathcal{K}}\to 0$
we have also $\|\mu_n-\mu\|_{\mathcal{F}}\to 0$, so that $\mu_n\Rightarrow \mu$ and $m_1(\mu_n)\to m_1(\mu)$. Assume now that  $\mu_n\Rightarrow \mu$, as $n\to \infty$ and $m_1(\mu_n)\to m_1(\mu)$. Then $X_n$ defined above converges to $X$ in $L^1(P)$ norm. In particular for any function $f$ with Lipschitz constant not greater than $1$ we have
\begin{equation}
|\mu_n(f)-\mu(f)|=|E\left[f(X_n)\right]-E\left[f(X)\right]|\leq E\left[|f(X_n)-f(X)|\right]\leq E|X_n-X|\to 0
\end{equation}
as $n\to \infty$, which means that we have also convergence in $\|\cdot\|_{\mathcal{K}}$ norm, which completes the proof.
\end{proof}
\begin{rem} The result above in not unexpected. For given measure $\mu\in D$ define $\bar{\mu}(A):=\int\limits_A x\mu(dx)$ for Borel measurable set $A$. Then compactness of the closure of the sequence $\left\{\bar{\mu}_n\in D\right\}$ is by Theorem 6.2 of \cite{1 billingsley} equivalent to the tightness of measures $\left\{\bar{\mu}_n\right\}$, which is equivalent to \eqref{unint1}.
\end{rem}
\begin{coll}\label{impcor}
Whenever $D\ni \mu_n\Rightarrow \mu$ and $\sup_n m_\beta(\mu_n)<\infty$ for some $\beta>1$ we have \linebreak $\|\mu_n-\mu\|_{\mathcal{K}}\to 0$ as $n\to \infty$.
\end{coll}
\begin{proof}
It is clear that
\begin{equation}
\sup_n \int\limits_M\limits^\infty x \mu_n(dx)\leq \sup_n \Big(\int\limits_0\limits^\infty x^\beta \mu_n(dx)\Big)^{1\over \beta} \Big(\int\limits_0\limits^\infty 1_{x\geq M}\mu_n(dx)\Big)^{\beta-1\over \beta}.
\end{equation}
Now $\int\limits_0\limits^\infty 1_{x\geq M}\mu_n(dx)\leq {1\over M}$, so that $\mu_n$ is uniformly integrable and it remains to use Theorem \ref{unif}.\linebreak
\end{proof}
Before we formulate next theorem we define metric $d_r$ in the space of probability measures defined on $\mathbb{R}_+$ with finite $r$-th  moments, where $r\in [1,2)$. Namely for probability measures  $\mu$ and $\nu$ such that $m_r(\mu)<\infty$ and $m_r(\nu)<\infty$ let
\begin{equation}
d_r(\mu,\nu):=\inf\left\{\left(E(|X-Y|^r)\right)^{1\over r}\right\},
\end{equation}
where infimum taken over probability measures $P$ on $\mathbb{R}_+^2$ such that their marginals are $\mu$ and $\nu$ respectively. We have

\begin{thm}\label{Rio} For $\mu,\nu\in D$ such that $m_r(\mu)<\infty$ and $m_r(\nu)<\infty$ with $r\in (1,2)$ we have
\begin{equation}\label{Rioin}
\|\mu-\nu\|_{\mathcal{K}}\leq 2(2\|\mu-\nu\|_r)^{1\over r}.
\end{equation}
\end{thm}
\begin{proof}
By \cite{Rio} we have that $d_r(\mu,\nu)\leq 2(2\|\mu-\nu\|_r)^{1\over r}$. Clearly $d_1(\mu,\nu)\leq d_r(\mu,\nu)$. Since by Theorem 20.1 of \cite{Dudley} $d_1(\mu,\nu)=\|\mu-\nu\|_{\mathcal{K}}$ we obtain \eqref{Rioin}.
\end{proof}

We recall now Kantorovich-Rubinstein maximum principle for our metric $\|\cdot\|_{\mathcal{K}}$, see Corollary 6.2 of \cite{2 rachev}
\begin{thm}\label{A2}
For probability measures $\mu,\nu$ defined on $\mathbb{R}_+$ there exists $f_0\in {\mathcal{K}}$ such that
\begin{equation}\label{equality}
\|\mu-\nu\|_{\mathcal{K}}=\langle f_0,\mu-\nu\rangle.
\end{equation}
Moreover when $f_0\in {\mathcal{K}}$ satisfies \eqref{equality} for measures $\mu\neq \nu$ defined on $\mathbb{R}_+$ then there are two different points $x_1,x_2\in \mathbb{R}_+$ such that $|f_0(x_1)-f_0(x_2)|=|x_1-x_2|$.
\end{thm}

Finally we recall now some known results related with ordinary differential equations in Banach spaces. For details see \cite{crandall}.

Let $( E, \|\cdot\|)$ be a Banach space and let $\tilde{D}$
 be a closed, convex, nonempty subset of $E$. In the space $E$ we
 consider {\it an evolutionary differential equation}\index{evolutionary differential equation}
\begin{equation}\label{s4.1a15}
    \frac{du}{dt}= - u + \tilde{P}\,u\qquad\text{for}\qquad t\in {\mathbb
    R}_{+}
\end{equation}
with the initial condition
\begin{equation}\label{s4.1a16}
    u(0)=u_{0}, \qquad u_{0}\in \tilde{D},
\end{equation}
where $\tilde{P} : \tilde{D} \to \tilde D$ is a given operator.

Function $u: {\mathbb R}_{+} \to E$ is called {\it a solution to the problem} (\ref{s4.1a15}), (\ref{s4.1a16}) if it is strongly
differentiable on ${\mathbb R}_{+}$, $u(t)\in \tilde{D}$ for all
$t\in {\mathbb R}_{+}$ and $u$ satisfies relations
(\ref{s4.1a15}), (\ref{s4.1a16}).

We have
\begin{thm}\label{sec4.twierdzenie1.3}
Assume that the operator  $\tilde{P}: \tilde{D}\to \tilde{D}$
satisfies Lipschitz condition
\begin{equation}\label{s4.1a17}
    \|\tilde{P}\,v - \tilde{P}\,w\| \leq l\,\|v - w\|\qquad\textnormal{for}\qquad u, w\in
    \tilde{D},
\end{equation}
where  $l$ is a nonnegative constant. Then for every   $u_{0}\in
\tilde{D}$ there exists a unique solution  $u$ to the problem
(\ref{s4.1a15}), (\ref{s4.1a16}).
\end{thm}
The standard proof of Theorem \ref{sec4.twierdzenie1.3} is
based on the fact, that function $u: \mathbb {R}_{+}\to
\tilde{D}$ is a solution to (\ref{s4.1a15}), (\ref{s4.1a16}) if and only if
it is continuous and satisfies the integral equation
\begin{equation}\label{s4.1a18}
u(t) = e^{-t}\,u_{0} +
\int\limits_{0}\limits^{t}\,e^{-(t-s)}\,\tilde{P}\,u(s)\,ds
\qquad\text{for}\qquad t\in {\mathbb
    R}_{+}.
\end{equation}
Due to completeness of $\tilde{D}$, the integral on the right hand
side is well defined and equation (\ref{s4.1a18}) may be solved by
the method of successive approximations.
Observe that, thanks to the properties of $\tilde{D}$, for every
$u_{0}\in \tilde{D}$ and for every continuous function $ u : {\mathbb
{R}}_{+} \to \tilde{D}$ the right hand side of (\ref{s4.1a18}) is
also a function with values in $\tilde{D}$.
The solutions of (\ref{s4.1a18}) generate a semigroup of operators
$(\tilde{P}^{\ t})_{t\geq 0}$ on $\tilde{D}$ given by the formula
\begin{equation}\label{s4.1a19}
    \tilde{P}^{\ t}\,u_{0} = u(t) \qquad\text{for}\qquad t\in {\mathbb
    R}_{+},\qquad u_{0}\in \tilde{D}.
\end{equation}

\end{document}